\documentclass{amsart}

\usepackage{url,amssymb,enumerate,colonequals}
\usepackage{tabu}
\usepackage{mathrsfs} 
\usepackage[section]{placeins}
\usepackage{MnSymbol}
\usepackage{extarrows}
\usepackage{lscape}
\usepackage[all,cmtip]{xy}

\usepackage[OT2,T1]{fontenc}
\usepackage{color}

\usepackage[
        colorlinks, citecolor=darkgreen,
        backref,
        pdfauthor={Samir Siksek},
]{hyperref}

\usepackage{comment}
\usepackage{multirow}

\newcommand{\F}{\mathbb{F}}

\newcommand{\PP}{\mathbb{P}}
\newcommand{\Q}{\mathbb{Q}}

\newcommand{\Z}{\mathbb{Z}}

\newcommand{\cP}{\mathscr{P}}
\newcommand{\cQ}{\mathcal{Q}}
\newcommand{\cA}{\mathcal{A}}
\newcommand{\cB}{\mathcal{B}}
\newcommand{\cC}{\mathscr{C}}
\newcommand{\cF}{\mathcal{F}}
\newcommand{\cN}{\mathcal{N}}

\newcommand{\fp}{\mathfrak{p}}
\newcommand{\fP}{\mathfrak{P}}

\newcommand{\OO}{\mathcal{O}}

\DeclareMathOperator{\GSp}{GSp}

\DeclareMathOperator{\Aut}{Aut}

\DeclareMathOperator{\Div}{Div}

\DeclareMathOperator{\End}{End}

\DeclareMathOperator{\Gal}{Gal}

\DeclareMathOperator{\Trace}{Trace}

\DeclareMathOperator{\rank}{rank}
\DeclareMathOperator{\re}{Re}

\DeclareMathOperator{\Sp}{Sp}
\DeclareMathOperator{\Spec}{Spec}

\renewcommand{\setminus}{-}

\newtheorem{thm}{Theorem}

\newtheorem{lem}[thm]{Lemma}
\newtheorem{conj}[thm]{Conjecture}
\newtheorem{cor}[thm]{Corollary}
\newtheorem{prop}[thm]{Proposition}

\theoremstyle{definition}
\newtheorem{example}[thm]{Example}

\theoremstyle{remark}

\definecolor{darkgreen}{rgb}{0,0.5,0}

\DeclareRobustCommand{\SkipTocEntry}[5]{}

\begin{document}

\title[Integral points on punctured abelian varieties]{
Integral points on punctured abelian varieties
}

\begin{abstract}
Let $A/\Q$ be an abelian variety such that $A(\Q)=0$.
	Let $\ell$ and $p$ be rational primes,
	such that $A$ has good reduction at $p$,
	and satisfying $\ell \equiv 1 \pmod{p}$ and
	$\ell \nmid \# A(\F_p)$. 
	Let $S$ be a finite set of rational primes.
	We show
	that $(A-0)(\OO_{L,S})=\emptyset$
	for 100\% of cyclic degree $\ell$ fields
	$L/\Q$, when ordered by conductor, or by absolute discriminant.
\end{abstract}

\author{Samir Siksek}

\address{Mathematics Institute\\
    University of Warwick\\
    CV4 7AL \\
    United Kingdom}

\email{s.siksek@warwick.ac.uk}

\date{\today}
\thanks{
Siksek is supported by the
EPSRC grant \emph{Moduli of Elliptic curves and Classical Diophantine Problems}
(EP/S031537/1).}
\keywords{Abelian varieties, cyclic fields, integral points}
\subjclass[2010]{Primary 11G10, Secondary 11G05}

\maketitle

\section{Introduction}

Let $L$ be a number field and write
$\OO_L$ for its ring of integers.
Let $S$ be a finite set of places of $L$,
and write $\OO_{L,S}$ for the ring of $S$-integers in $L$.
Let $A$ be an abelian variety over $L$.
A theorem of Silverman \cite{SilvermanIntegral}
asserts the existence of a very ample divisor
$D \in \Div(A)$ such that
the set of $S$-integral points $(A-D)(\OO_{L,S})$
is finite.
Write $0 \in A$ for  the origin. 
If $\dim(A)=1$, then the
finiteness of $(A-0)(\OO_{L,S})$
is a famous theorem of Siegel \cite[Section IX.3]{SilvermanAEC}. 
Little
is known about the integral points on $A-0$
for $\dim(A) \ge 2$.
A special case of the \emph{Arithmetic Puncturing Problem}
of Hassett and Tschinkel \cite[Problem 2.13]{HT}
asks whether the integral points on $A-0$
are potentially dense. This note explores
an obstruction to the existence of integral
points on $A-0$.

For a finite prime $\fP$ of $\OO_L$
we denote the residue field by $\F_\fP=\OO_L/\fP$,
and the completion of $L$ at $\fP$ by $L_\fP$.
If $A$ has good reduction at $\fP$
we will write
$A^1(L_\fP)$ for the kernel of the reduction
map $A(L_\fP) \rightarrow A(\F_\fP)$.

\begin{thm}\label{thm:main}
Let $K$ be a number field, 
and let $A$ be an abelian variety defined over $K$
	satisfying $A(K)=0$.
Let $\fp$
be a finite prime of $\OO_K$ of good reduction for $A$.
Let $L/K$ be an  extension of degree $m$.
Suppose that
	\begin{enumerate}[(i)]
\item $\fp$ is totally ramified in $L$;
\item $\gcd(\# A(\F_\fp),m)=1$.
\end{enumerate}
Then $A(L) \subseteq A^1(L_\fP)$
where 
 $\fP$ be the unique prime of $\OO_L$ above $\fp$.
	In particular, $(A\setminus 0)(\OO_{L,S})=\emptyset$,
	for any set of places $S$ not containing $\fP$.
\end{thm}
\begin{proof}[Proof for $L/K$ Galois]
The theorem is proved in Section~\ref{sec:proofthmmain}.
However, when $L/K$ is Galois, the theorem admits a shorter and more conceptual
proof, which we now give.
Recall that the inertia subgroup $I_\fP \subseteq \Gal(L/K)$
is by definition the subset of $\sigma \in \Gal(L/K)$
	such that $\sigma(\alpha) \equiv \alpha \pmod{\fP}$
	for all $\alpha \in \OO_L$. 
Since $\fp$ is totally ramified, 
we have $I_\fP=\Gal(L/K)$. We deduce that
$\sigma(Q) \equiv Q \pmod{\fP}$
for all $Q \in A(L)$ and all $\sigma \in \Gal(L/K)$.
Thus
\[
	\Trace_{L/K}(Q) \; = \; \sum_{\sigma \in \Gal(L/K)} \sigma(Q)  \; \equiv \; m Q \pmod{\fP}.
\]
	However, $\Trace_{L/K}(Q) \in A(K)=0$ by assumption. Thus $mQ \equiv 0 \pmod{\fP}$.
	Now, again as $\fp$ is totally ramified, $\F_\fP=\F_\fp$, and so
	$A(\F_\fP)=A(\F_\fp)$. By assupmotion (ii) we have $Q \equiv 0 \pmod{\fP}$
	completing the proof.
\end{proof}


\begin{cor}\label{cor:C}
Let $C/K$ be a curve of genus $\ge 1$,
	and let $Q_0 \in C(K)$.
	Let $J$ be the Jacobian of $C$
	and suppose $J(K)=0$.
Let $\fp$
be a finite prime of $\OO_K$ of good reduction for $C$.
Let $L/K$ be an extension of degree $m$.
Suppose that
	\begin{enumerate}[(i)]
\item $\fp$ is totally ramified in $L$;
\item $\gcd(\# J(\F_\fp),m)=1$.
\end{enumerate}
	Then $(C\setminus \{Q_0\})(\OO_{L,S})=\emptyset$
	for any set of places $S$ not containing $\fP$.
\end{cor}
\begin{proof}
	If $Q \in (C\setminus \{Q_0\})(\OO_{L,S})$
	then the linear equivalence class $[Q-Q_0]$ yields 
	an element of $(J-0)(\OO_{L,S})$,
contradicting Theorem~\ref{thm:main}.
\end{proof}

\begin{example}
Let $E/\Q$ be an elliptic curve 
with complex multiplication
by an order in an imaginary quadratic field $K$.
Let $p$ be a prime of good supersingular reduction for $E$,
and write $K_n$ for the $n$-th layer of the
anticyclotomic $\Z_p$-extension of $K$.
It is known \cite[Theorem 1.8]{GreenbergParkCity} 
that $E(K_n)$ has unbounded rank as $n \rightarrow \infty$.
	Indeed $\rank(E_{K_n})-\rank(E_{K_{n-2}})=2p^{n-1} (p-1)$
	for sufficiently large $n$.

Suppose now that $p$ is unramified in $K$. 
As $E/\F_p$ is supersingular, we know that $p$ is inert in $K$.
Write $\fp=p\OO_K$ for the unique prime of $\OO_K$ above $p$.
Since $E/\F_p$ is supersingular, $a_\fp(E) \equiv 0 \pmod{p}$
and so $\#E(\F_\fp) \equiv 1 \pmod{p}$. In particular, $p \nmid \#E(\F_\fp)$.

Let $n \ge 1$. 
By \cite[Theorem 1]{Iwasawa73},
the extension $K_n/K$ is unramified away from $\fp$.
We show that $\fp$ is totally ramified in $K_n$.
Let $\fP$ be a prime ideal of $\OO_{K_n}$
above $\fp$, and let $I_\fP \subseteq \Gal(K_n/K)$ 
be the inertia group. As $K_n/K$ is cyclic,
$I_\fP$ is a normal subgroup.
In particular, $I_\fP=I_{\fP^\prime}$ for
any other prime ideal $\fP^\prime$
of $\OO_{K_n}$ above $\fp$.
It follows that the fixed field $K_n^{I_\fP}$
is an unramified cyclic extension of $K$. 
However, $K$ is the CM field of an elliptic curve
	defined over $\Q$ and so \cite[Theorem II.4.3]{SilvermanAdvanced} it 
has class number $1$. Therefore $K_n^{I_\fP}=K$,
implying $I_\fP=\Gal(K_n/K)$,
and so $\fp$ is totally ramified in $K$.

Finally we suppose that $E(K)=0$.
It now follows from Theorem~\ref{thm:main}
that $(E-0)(\OO_{K_n})=\emptyset$ for all $n \ge 1$,
despite the fact that the rank of $E(K_n)$
is unbounded as $n \rightarrow \infty$.

As a very concrete example of the above, let $E/\Q$ be the elliptic
curve with Cremona label \texttt{432a1} and Weierstrass model
\[
	E \; : \; Y^2=X^3-16.
\]
This has conductor $432=2^4 \times 3^3$, and 
has CM by the ring of integers of $K=\Q(\sqrt{-3})$.
	We checked using the computer algebra system \texttt{Magma}
	\cite{Magma}
	that $E(K)=0$. Let $p$ be an odd prime $\equiv 2 \pmod{3}$.
Then $p$ is a prime of good supersingular reduction for $E$,
and for every $n \ge 1$, we have $(E-0)(\OO_{K_n})=\emptyset$
where $K_n$ is the $n$-th layer of anticyclotomic $\Z_p$-extension of $K$.
\end{example}

Throughout the paper $\zeta_r$ denotes a primitive $r$-th root of $1$.
\begin{cor}\label{cor:cyclotomic}
Let $A/\Q$ be an abelian variety satisfying $A(\Q)=0$,
and write $\cN_A$ for the conductor of $A$. Let
\[
	R_A \; =\; \{  \text{$p \nmid \cN_A$ is prime}\; : \; \gcd(p(p-1),\#A(\F_p))=1\}.
\]
Then $(A-0)(\Z[\zeta_{p^n}])=\emptyset$ for all $p \in R_A$ and $n \ge 1$.
\end{cor}
\begin{proof}
Let $p \in R_A$ and write $L=\Z[\zeta_{p^n}]$. 
Then $p$ is totally ramified in $L$, and as $p \nmid \cN_A$,
it is a prime of good reduction for $A$.
Moreover, $[L:\Q]=p^{n-1}(p-1)$ is coprime
to $\#A(\F_p)$. The conclusion follows from Theorem~\ref{thm:main}.
\end{proof}
The set $R_A$ can be finite or empty. For example if $A$ has a rational
point of order $2$ then
$2 \mid \#A(\F_p)$
for all odd primes of good reduction, and so 
$R_A \subseteq \{2\}$ in this case. In a forthcoming paper
we provide heuristic and experimental evidence that $R_A$ has positive density
under some conditions on $A$. For now we 
content ourselves with two examples.

\begin{example}\label{ex:elliptic}
Let $E/\Q$ be the elliptic curve with \texttt{LMFDB} \cite{LMFDB}
	label \texttt{67.a1} and Cremona label \texttt{67a1}.
	This has Weierstrass model
	\begin{equation}\label{eqn:E}
	E \; : \; 
	Y^2+Y=X^3+X^2-12X-21,
	\end{equation}
	conductor $67$ and Mordell--Weil group $E(\Q)=0$.
	By Corollary~\ref{cor:cyclotomic}, the affine Weiestrass model \eqref{eqn:E}
	does not have any $\Z[\zeta_{p^n}]$-points for the values of $p \in R_E$.
	For a positive integer $N$ we shall write
	$[1,N]$ for the interval consisting
	of integers up to $N$.
	A short \texttt{Magma} computation
	shows that 
	\[
	\begin{split}
		R_E \cap [1,1000] &= \{
	2,\; 17,\; 19,\; 23,\; 47,\; 59,\; 89,\; 107,\; 
		127,\; 149,\; 151,\; 157,\; 163,\\ 
		& \quad 173,\; 
		193,\; 199,\; 227,\; 
		257,\; 283,\; 359,\; 421,\; 431,\;
		449,\; 479,\; \\
		& \quad 491,\; 509,\; 569, \;
		 601,\;
		613,\; 617,\; 659,\; 691,\;
		719,\; 773,\; 821,\\ 
		& \quad 823,\; 
		827,\; 839,\;
		 881,\; 887,\; 911,\; 947,\; 953,\; 971,\; 977 
		\}.
	\end{split}
	\]
	Table~\ref{table1}	gives some statistics.	

\begin{table}[!htbp]
\begin{centering}
{\tabulinesep=1.2mm
\begin{tabu}{|c|c|c|c|}
\hline\hline
\rule[-2.0ex]{0pt}{5ex}
$k$ & $\# R_E \cap [1,10^k]$ & $\pi(10^k)$ & $(\# R_E \cap [1,10^k])/\pi(10^k)$ (4 d.p.)\\
\hline\hline
$2$ & $7$ &  $25$ & $0.2800$\\
$3$ & $45$ & $168$ & $0.2679$\\
$4$ & $297$ & $1229$ & $0.2417$\\
$5$ & $2309$ & $9592$ & $0.2407$\\
$6$ & $19060$ & $78498$ & $0.2428$\\
$7$ & $160958$ & $664579$ & $0.2422$\\
$8$ & $1395958$ & $5761455$ & $0.2423$\\
\hline
\end{tabu}
}
\end{centering}
\caption{
We write $\pi(N)$ for the number of primes $\le N$.
This table gives statistics for $R_E \cap [1,10^k]$
for $2 \le k \le 8$, where $E$ is the elliptic
curve \texttt{67a1}.
}
\label{table1}
\end{table}

\end{example}
\begin{example}\label{ex:gen2}
Let $C/\Q$ be the genus $2$ curve with \texttt{LMFDB}
label \texttt{8969.a.8969.1}
having affine Weierstrass model
\begin{equation}\label{eqn:Weierstrass}
	C \; : \; y^2 + (x + 1) y = x^5 - 55 x^4 - 87 x^3 - 54 x^2
    - 16 x - 2.
\end{equation}
We take $A=J$ to be the Jacobian of $C$.
According to the \texttt{LMFDB}, 
$J$ is absolutely simple, and $J(\Q)=\{0\}$.
The conductor $\cN_J=8969$ which is prime.
We note that $C$ has a rational point
at $\infty$, and thus $C(\Q)=\{\infty\}$.
By Corollary~\ref{cor:cyclotomic}, 
$(J-0)(\Z[\zeta_{p^n}])=\emptyset$
for all $p \in R_J$,
and so the affine Weierstrass model in \eqref{eqn:Weierstrass}
has no $\Z[\zeta_{p^n}]$-points for all $n \ge 1$.
A short \texttt{Magma} computation gives 
	\[
	\begin{split}
		R_J \cap [1,1000] & =\{
11,\; 13,\; 43,\; 79,\; 149,\; 163,\; 223,\; 227,\; 269,\; 353,\;
	367,\; 443,\\ 
		& \quad 523,\; 
		 593,\; 641,\; 683,\; 
743,\; 769,\; 797,\; 887,\; 929,\; 941,\; 991
\}.
	\end{split}
	\]
Note $\# R_J \cap [1,1000]=23$, $\pi(1000)=168$, and so $(\# R_J \cap [1,1000])/\pi(1000) \approx 0.137$.
\end{example}

\bigskip

Our next theorem concerns abelian varieties
$A$ defined over $\Q$ with trivial Mordell--Weil
group; i.e. $A(\Q)=0$. Let $\ell$ be a rational
prime, and let $S$ be a finite set of rational
primes. The theorem states that, under an additional
hypothesis, $(A-0)(\OO_{L,S})=\emptyset$
for 100\% of degree $\ell$ cyclic extensions $L/\Q$,
ordered by conductor. 
Here $\OO_{L,S}$ denotes $\OO_{L,T}$
where $T$ is set of places of $L$
above the rational primes belonging to $S$.
We denote by $\zeta_\ell$ a fixed primitive
$\ell$-th root of $1$, and by $A[\ell]$
the $\ell$-torsion subgroup of $A(\overline{\Q})$.
We observe that $\Q(\zeta_\ell) \subseteq \Q(A[\ell])$
(for a proof see Lemma~\ref{lem:WeilPairing} below).
We shall write
\begin{equation}\label{eqn:GH}
	G_\ell(A)=\Gal(\Q(A[\ell])/\Q),
	\qquad
	H_\ell(A)=\Gal(\Q(A[\ell])/\Q(\zeta_\ell)).
\end{equation}
We note that $H_\ell(A)$ is a normal subgroup of $G_\ell(A)$.
We also write
\begin{equation}\label{eqn:cC}
	\cC_\ell(A)=\{\sigma \in H_\ell(A) \; : \;
	\text{$\sigma$ acts freely on $A[\ell]$} \}.
\end{equation}

\begin{thm}\label{thm:cyclic}
Let $\ell$ be a rational prime.
Let
 $A$ be an abelian variety 
defined over $\Q$.
Suppose that
\begin{enumerate}[(i)]
\item $A(\Q)=0$;
\item $\cC_\ell(A) \ne \emptyset$.
\end{enumerate}
For $X>0$, let $\cF^{\mathrm{cyc}}_\ell(X)$ be set of cyclic number fields $L$
of degree $\ell$ and conductor at most $X$. 
Let $S$ be a finite set of rational primes.
Then 
\[
	\frac{\# \{L \in \cF^{\mathrm{cyc}}_\ell(X) \; : \; 
	(A \setminus 0)(\OO_{L,S}) \ne \emptyset\}}{\#\cF^{\mathrm{cyc}}_\ell(X) } \; = \; O\left(\frac{1}{(\log{X})^{\gamma}} \right)
\]
as $X \rightarrow \infty$, where
	\[
		\gamma \; = \; \frac{\# \cC_\ell(A)}{\# H_\ell(A)}.
	\]
\end{thm}

\noindent \textbf{Remark.}
Let $L/\Q$ be cyclic of prime degree $\ell$.
Write $N$ for the conductor of $L$,
and $\Delta$ for its absolute discriminant.
It easily follows from the discriminant-conductor formula
\cite[Theorem 3.11]{Washington}
that $\Delta=N^{\ell-1}$. 
The conclusion of Theorem~\ref{thm:cyclic}
is therefore unchanged if 
instead we let 
$\cF^{\mathrm{cyc}}_\ell(X)$
be the set of
cyclic degree $\ell$ number fields
with absolute discrminant at most $X$.

\bigskip

Condition (ii) of Theorem~\ref{thm:cyclic},
in its present form, is computationally unfriendly. 
The following lemma simplifies the task of
checking condition (ii).

\begin{lem}\label{lem:equiv}
Let $p \ne \ell$ be a rational prime of good
reduction for $A$. Write $\sigma_p \in G_\ell(A)$
for a Frobenius element at $p$.
	\begin{enumerate}[(a)]
		\item $\sigma_p \in H_\ell(A)$
			if and only if 
			$p \equiv 1 \pmod{\ell}$.
		\item $\sigma_p \in \cC_\ell(A)$ if and only
	if $p \equiv 1 \pmod{\ell}$
	and $ \ell \nmid \# A(\F_p)$.
	\end{enumerate}
\end{lem}
\begin{proof}
Let $p \ne \ell$ be a prime of good reduction for $A$.
Recall that the isomorphism
$\Gal(\Q(\zeta_\ell)/\Q) \cong (\Z/\ell \Z)^\times$
sends the Frobenius element at a prime $q \ne \ell$ to 
the congruence class of $q$ modulo $\ell$.
However, $\Gal(\Q(\zeta_\ell)/\Q) \cong G_\ell(A)/H_\ell(A)$,
	thus $\sigma_p \in H_\ell(A)$
	if and only if $p \equiv 1 \pmod{\ell}$.
Moreover, we know 
\cite[Theorem 19.1]{Milne} that $\#A(\F_p)= P_p(1)$
where $P_p$ is the characteristic polynomial
of Frobenius at $p$ acting on the $\ell$-adic Tate module
	$T_\ell(A)$. Thus $\ell \mid \#A(\F_p)$
	if and only if $1$ is a root
	of $\overline{P_p}(X) \in \F_\ell[X]$.
	This is equivalent to $1 \in \F_\ell$
	being an eigenvalue for the 
	action of $\sigma_p$ on the $\F_{\ell}$-vector
	space $A[\ell]$, which is equivalent to
	$\sigma_p$ failing to act freely on $A[\ell]$.
\end{proof}
Lemma~\ref{lem:equiv} gives a computational method
for verifying condition (ii) of Theorem~\ref{thm:cyclic}
for a given prime $\ell$: 
all we need to do is produce a prime $p \equiv 1 \pmod{\ell}$
such that $\ell \nmid \# A(\F_p)$.
To check that 
condition (ii) holds for all primes $\ell$, or all but 
finitely many primes $\ell$, the following lemma
can be useful.
\begin{lem}\label{lem:PP}
Let $A/\Q$ be a principally polarized abelian
variety of dimension $d$. Let $\ell$ be a rational prime
and write 
\[
	\overline{\rho}_{A,\ell} \; : \; \Gal(\overline{\Q}/\Q) \rightarrow \GSp_{2d}(\F_\ell)
\]
for the mod $\ell$ representation of $A$. Suppose $\overline{\rho}_{A,\ell}$ 
	is surjective. Then $\cC_\ell(A) \ne \emptyset$.
\end{lem}
\begin{proof}
Suppose $\overline{\rho}_{A,\ell}$ is surjective. The map $\overline{\rho}_{A,\ell}$
	factors through $G_\ell(A)$. The image 
	of $H_\ell(A) \subseteq G_\ell(A)$
	is $\Sp_{2d}(\F_\ell)$. An element $\sigma \in H_\ell(A)$
acts freely on $A[\ell]$ if and only if its image in $\Sp_{2d}(\F_\ell)$
is a matrix with none of the eigenvalues equal to $1 \in \F_\ell$. 
	All that remains
is to specify such a matrix $M \in \Sp_{2d}(\F_\ell)$. If $\ell \ne 2$ we may take $M=-I_{2d}$
where $I_{2d}$ is the $2d \times 2d$ identity matrix.
If $\ell = 2$ then we may take 
\[
M=
\begin{pmatrix}
1 & 1 & 0 & 0 & \cdots & 0 & 0\\
1 & 0 & 0 & 0 & \cdots & 0 & 0\\
0 & 0 & 1 & 1 & \cdots & 0 & 0 \\
0 & 0 & 1 & 0 & \cdots & 0 & 0\\
\vdots & \vdots & \vdots & \vdots & \ddots & \vdots & \vdots\\
0 & 0 & 0 & 0 & \cdots & 1 & 1 \\
0 & 0 & 0 & 0 & \cdots & 1 & 0\\
\end{pmatrix}.
\]
\end{proof}

It follows, thanks to the following theorem of Serre \cite[Theorem 3]{serreIV},
that condition (ii) of Theorem~\ref{thm:cyclic} is satisfied for all sufficiently large
 $\ell$ subject to some further assumptions on $A$.
\begin{thm}[Serre]
Let $A$ be a principally polarized abelian variety of dimension $d$,
defined over $\Q$.
Assume that $d = 2$, $6$ or $d$ is odd and furthermore assume that
$\End_{\overline{\Q}}(A)=\Z$. Then there exists a bound $B_{A}$ such that
for all primes $\ell > B_{A}$ the representation $\overline{\rho}_{A,\ell}$
is surjective.
\end{thm}

\begin{example}
We return to the elliptic curve $E$ in Example~\eqref{ex:elliptic}.
	We noted previously that $E(\Q)=0$.
	According to the \texttt{LMFDB},
	$\overline{\rho}_{E,\ell}$ is surjective
	for all primes $\ell$. It follows from
	Lemma~\ref{lem:PP} and Theorem~\ref{thm:cyclic}
	that for any prime $\ell$, 
	and any fixed set of rational primes $S$,
	the Weierstrass model \eqref{eqn:E}
	does not have $\OO_{L,S}$-integral points,
	for 100\% of cyclic degree $\ell$ number
	fields $L$.
\end{example}
\begin{example}
We return to the genus $2$ curve $C$ in Example~\ref{ex:gen2}
and to its Jacobian $J$.
	We observed previously that $J(\Q)=0$.
In particular, $J$ satisfies hypothesis (i) of Theorem~\ref{thm:cyclic}.
Moreover, $J$ is semistable as its conductor $\cN_J=8969$ is prime.
Using the method in \cite{ALS}, \cite{Dieulefait} (which is
particularly suited to semistable Jacobians),
we checked that
$\overline{\rho}_{J,\ell}$ is surjective
for $\ell \ge 5$, $\ell \ne 8969$.
	Thus, by Lemma~\ref{lem:PP},
	the Jacobian $J$ satisfies hypothesis (ii) of Theorem~\ref{thm:cyclic}
for those primes.
For $\ell=2$, $3$, $8969$ we choose 
$p=5$, $7$, $17939$ respectively (all three satisfying $p \equiv 1 \pmod{\ell}$), and find
\[
	\#J(\F_5)=15, \qquad \#J(\F_7)=32, 
	\qquad \# J(\F_{17939})
	=
	317816600
	=2^3 \times 5^2 \times 1589083,
\]
	so, by Lemma~\ref{lem:equiv},
	hypothesis (ii) of the theorem is satisfied
for $\ell=2$, $3$ and $8969$.
	It follows from Theorem~\ref{thm:cyclic}
	that for all primes $\ell$, 
and any finite set of primes $S$,
we have $(J-0)(\OO_{L,S})=\emptyset$
for $100\%$ of cyclic degree $\ell$ number fields $L$.
We conclude that $(C-\infty)(\OO_{L,S})=\emptyset$
for $100\%$ of cyclic degree $\ell$ number fields $L$.
\end{example}

\bigskip

The paper is organized as follows. In
Section~\ref{sec:traces}, we study traces
on abelian varieties over totally ramified
local extensions. In Section~\ref{sec:proofthmmain}
we prove Theorem~\ref{thm:main}.
Section~\ref{sec:counting} is devoted
to counting cyclic fields of prime degree $\ell$
such that the conductor is divisible only
by primes belonging a certain \lq regular\rq\
set. Section~\ref{sec:proofthmcyclic}
gives a proof of Theorem~\ref{thm:cyclic}.

\bigskip

We would like to thank Ariyan Javanpeykar
for useful discussions, and for bringing the 
Arithmetic Puncturing Problem to 
our attention.

\section{Traces over totally ramified local extensions}\label{sec:traces}
In this section, we let $p$ be a rational prime, and $K$
a finite extension of $\Q_p$, and
$L/K$
a totally ramified extension of 
finite degree $m$. 
Let $\pi$ and $\Pi$ be uniformizing elements for $K$
and $L$ respectively. 
Let $M/K$ 
be the Galois closure of $L/K$. 
Let $\lvert \, \cdot\, \rvert$ denote
the absolute value on these fields
normalised so that $\lvert p \rvert=p^{-1}$.
Write $\sigma_1,\dotsc,\sigma_m$ for
the distinct embeddings $L \hookrightarrow M$
satisfying $\sigma_i(a)=a$ for $a \in K$,
where $\sigma_1$ is the trivial embedding
$\sigma_1(\alpha)=\alpha$ for $\alpha \in L$.
\begin{lem}\label{lem:dodgy1}
	Let $\alpha \in \OO_L$. Then $\lvert \sigma_i(\alpha)-\alpha \rvert<1$
for $i=1,\dots,m$.
\end{lem}
\begin{proof}
As $L/K$ is totally ramified
we have $\OO_L/\Pi=\OO_K/\pi$. Hence there is some
$a \in \OO_K$ such that $\alpha \equiv a \pmod{\Pi}$.
It follows that $\lvert \alpha-a \rvert<1$.
Now, as each $\sigma_i$ is the restriction to 
$L$ of an automorphism of $M/K$,
the differences $\alpha-a$ and $\sigma_i(\alpha)-a$
are conjugate over $K$. Therefore \cite[page 119]{Cassels},
$\lvert \sigma_i(\alpha)-a \rvert=\lvert \alpha-a \rvert<1$.
By the ultrametric property of non-archimedean absolute
values, $\lvert \sigma_i(\alpha) - \alpha \rvert<1$.
\end{proof}

\begin{lem}\label{lem:dodgy2}
Let $A/K$ be an abelian
variety having good reduction.
	Let $Q \in A(L)$. Then
	\begin{equation}\label{eqn:trace}
		\Trace_{L/K} Q \; \equiv \; m Q \pmod{\Pi}.
	\end{equation}
\end{lem}
\begin{proof}
We first prove \eqref{eqn:trace} under
the additional assumption that $L=K(Q)$.
Let $Q_i=\sigma_i(Q) \in A(M)$ with $Q=Q_1$. 
The assumption $L=K(Q)$ ensures
$Q_1,\dotsc,Q_m$ are distinct
as well as being a single Galois
orbit over $K$, and so allows
us to interpret the $m$-tuple
$\{Q_1,\dotsc,Q_m\}$ as a closed
$K$-point on $A$.
As $A$ has good reduction, it extends to an abelian
scheme $\cA$ over $\Spec(\OO_K)$, and the closed
	$K$-point $\{Q_1,\dotsc,Q_m\}$
	extends to a $\Spec(\OO_K)$-point 
	on $\cA$ that we denote by $\cQ$.
We take an affine patch $\Spec(\OO_K[x_1,\dotsc,x_n]/(f_1,\dotsc,f_r))$
	of $\cA$
	containing $\cQ$. In this patch we can identify $Q$ with 
	a point $Q=(q_1,\dotsc,q_n) \in \OO_L^n$
	satisfying $f_1(q_1,\dotsc,q_n)=\cdots=f_r(q_1,\dotsc,q_n)=0$.
	Then $Q_i=(\sigma_i(q_1),\dotsc,\sigma_i(q_n))$.
	Let $\varpi$ be a uniformizing element for $M$.
	Then $\sigma_i(q_j) \equiv q_j \pmod{\varpi}$
	by Lemma~\ref{lem:dodgy1}. Thus $Q_i \equiv Q \pmod{\varpi}$.
	Hence
	\[
		\Trace_{L/K} Q = \sum_{i=1}^m Q_i 
		\equiv m Q
		\pmod{\varpi}.
	\]
	Now \eqref{eqn:trace} follows as both $\Trace_{L/K} Q$ and $mQ$
	belong to $A(L)$.

	For the general case, let $L^\prime=K(Q) \subseteq L$, $m^\prime=[L^\prime:K]$
	and $\Pi^\prime$ be a uniformizer for $L^\prime$. Then,
	by the above, 
	\[
	\Trace_{L^\prime/K} Q \equiv m^\prime Q \pmod{\Pi^\prime}.
	\]
	Therefore
	\[
	\Trace_{L/K} Q \; =\;
 \Trace_{L/L^\prime} \left( \Trace_{L^\prime/K} Q \right) 
\; \equiv\;  [L:L^\prime]  \cdot m^\prime Q \; = \; mQ \pmod{\Pi^\prime}.
	\]
	The lemma follows as $\Pi \mid (\Pi^\prime \cdot \OO_L)$.
\end{proof}


\section{Proof of Theorem~\ref{thm:main}}\label{sec:proofthmmain}
With notation and assumptions as in the
statement of Theorem~\ref{thm:main}, let
$Q \in A(L)$. 
Then $\Trace_{L/K}(Q) \in A(K)$. 
However, by assumption, $A(K)=0$,
and so
$\Trace_{L/K}(Q)=0$. 
By Lemma~\ref{lem:dodgy2}
we have
\[
	mQ \equiv \Trace_{L/K}(Q)  \pmod{\fP}.
\]
Thus $mQ \equiv 0 \pmod{\fP}$.
But, since $\fp$ is totally ramified, $\F_\fP=\F_\fp$,
	and so $A(\F_\fP)=A(\F_\fp)$. It follows from 
assumption (ii) of the statement of the theorem
that $Q \equiv 0 \pmod{\fP}$. Thus $Q \in A^1(L_\fP)$ completing the proof.

\section{Counting Cyclic Fields}\label{sec:counting}
Let $\PP$ be the set of prime numbers
and let $\cP \subseteq \PP$. Following Serre \cite{Serre},
we call $\cP$ \textbf{regular of density $\alpha>0$}
if
\begin{equation}\label{eqn:regular}
\sum_{p \in \cP} \frac{1}{p^s} \; =\; \alpha \cdot \log\left( \frac{1}{s-1} \right) \; + \; \theta_A(s)
\end{equation}
where $\theta_A$ extends to a holomorphic function on $\re(s) \ge 1$.
We call the set $\cP$ \textbf{Frobenian of density $\alpha>0$}
if there exists a finite Galois extension $L/\Q$
and a subset $\cC$ of $G=\Gal(L/\Q)$, such that
\begin{itemize}
\item $\cC$ is a union of conjugacy classes in $G$;
\item $\alpha=\# \cC/\# G$;
\item for every sufficiently large prime $p$,
we have $p \in \cP$ if and only if $\sigma_p \in \cC$
where $\sigma_p$ is a Frobenius element of $G$
corresponding to $p$. 
\end{itemize}
By the Chebotarev Density Theorem \cite[Proposition 1.5]{Serre},
if $\cP$ is Frobenian of density $\alpha>0$
then it is regular of density $\alpha>0$.

Let $\ell$ be a rational prime, and let
\begin{equation}\label{eqn:PPell}
	\PP_\ell=\{ \ell \} \cup 
	\{p \; : \; \text{$p$ is prime $\equiv 1 \pmod{\ell}$}\}.
\end{equation}
The purpose of this section is to prove the following proposition
which will be needed for the proof of Theorem~\ref{thm:cyclic}.
\begin{prop}\label{prop:Ikehara}
Let $\cP \subseteq \PP_\ell$ 
and suppose $\cP$ is regular of density $\alpha>0$.
For $X>0$ let
$\cF^{\mathrm{cyc}}_{\cP,\ell}(X)$ be the set of number fields $L$ such that
\begin{enumerate}[(i)]
\item $L$ is cyclic of degree $\ell$; 
\item the conductor of $L$ is divisible
only by primes belonging to $\cP$:
\item the conductor of $L$ is at most $X$.
\end{enumerate}
There is some $c>0$ such that
\[
	\# \cF^{\mathrm{cyc}}_{\cP,\ell}(X) \, \thicksim \,  c \cdot \frac{X}{(\log{X})^{1-\beta}}, 
\]
as $X \rightarrow \infty$, where
$\beta=\alpha \cdot (\ell-1)$.
\end{prop}
\noindent \textbf{Remark.}
By Lemma~\ref{lem:Nn} below,
$\cF^\mathrm{cyc}_{\PP_\ell,\ell}(X)=\cF^{\mathrm{cyc}}_\ell(X)$ is the set of all degree $\ell$
cyclic number fields of conductor at most $X$. 
By Dirichlet's Theorem, 
the set $\PP_\ell$ is regular of density $1/(\ell-1)$.
The proposition is saying in this case that
\[
	\# \cF^{\mathrm{cyc}}_{\ell}(X) \; \sim \; c X
\]
as $X \rightarrow \infty$. This is in fact a theorem
of Urazbaev \cite{Urazbaev}. A proof can also 
be found in \cite[Section 2.2]{Pollack},
and a generalization to more general abelian
extensions in \cite{Wright}.
Lemmas~\ref{lem:rank},~\ref{lem:Mn},~\ref{lem:qalpha},~\ref{lem:CRT},
~\ref{lem:Nn} below are in essence well-known,
and can be found in some form or other 
scattered across the literature, e.g.
\cite[Section 1]{Maki}, \cite[Section 2.2]{Pollack}.
It however seemed more convenient to prove them
from scratch.

\bigskip

Let $G$ be a finite abelian group, for now written additively.
Let $\ell$ be a prime. We define the \textbf{$\ell$-rank of $G$}
to be the dimension of the $\F_\ell$-vector space $G/\ell G$.
\begin{lem}\label{lem:rank}
Let $r$ be the $\ell$-rank of $G$.
Then the number of subgroups of index $\ell$ in $G$ 
is $(\ell^r-1)/(\ell-1)$.
\end{lem}
\begin{proof}
Any subgroup $H$ of $G$ of index $\ell$ contains $\ell G$.
Thus there is a $1-1$ correspondence between
subgroups of index $\ell$ in $G$ and subgroups of index
$\ell$ in $G/\ell G$,
or equivalently $\F_\ell$-subspaces
of $G/\ell G$ of codimension $1$.
But, regarded as an $\F_\ell$-vector space,
 $G/\ell G$ is isomorphic to $\F_\ell^r$. 
The codimension $1$ subspaces of $\F_\ell^r$
correspond to
points in $\check{\PP}^{r-1}(\F_\ell)$,
where $\check{\PP}^{r-1}$ denotes
the projective space dual to $\PP^{r-1}$.
However, $\check{\PP}^{r-1} \cong \PP^{r-1}$.
The lemma follows.
\end{proof}

Let $M(n)$ denote the number of degree $\ell$ cyclic 
fields contained in $\Q(\zeta_n)$. Let $N(n)$ 
denote the number of degree $\ell$ cyclic fields
of conductor $n$. Then
\begin{equation}\label{eqn:condsum}
	M(n)=\sum_{d \mid N} N(d).
\end{equation}
\begin{lem}\label{lem:Mn}
Let $n$ be a positive integer.
Write $r_\ell(n)$ for the $\ell$-rank of $(\Z/n\Z)^\times$.
Then
\[
	M(n)=\frac{\ell^{r_\ell(n)}-1}{\ell-1}.
\]
\end{lem}
\begin{proof}
By the Galois correspondence, $M(n)$ 
is the number of index $\ell$
subgroups in 
\[
	\Gal(\Q(\zeta_n)/\Q) \cong (\Z/n\Z)^\times.
\]
The lemma follows from
Lemma~\ref{lem:rank}.
\end{proof}
\begin{lem}\label{lem:qalpha}
Let $q$ be a prime and $\alpha \ge 1$. Then
\[
r_\ell(q^{\alpha})=
\begin{cases}
1 & \text{if $q \equiv 1 \pmod{\ell}$}\\
1 & \text{if $q=\ell \ne 2$ and $\alpha \ge 2$}\\
1 & \text{if $q=\ell=2$ and $\alpha=2$}\\
2 & \text{if $q=\ell=2$ and $\alpha \ge 3$}\\
0 & \text{in all other cases}.
\end{cases}
\]
\end{lem}
\begin{proof}
If $q \ne 2$ then $(\Z/q^{\alpha} \Z)^\times$
is cyclic of order $(q-1)q^{\alpha-1}$.
Thus $r_\ell(q^{\alpha})=0$ unless
$q \equiv 1 \pmod{\ell}$ or $q=\ell$ and $\alpha \ge 2$,
in which case $r_\ell(q^{\alpha})=1$.

Suppose $q=2$. Then 
\[
(\Z/2^\alpha \Z)^\times \cong
\begin{cases}
0 & \alpha=1\\
\Z/2\Z & \alpha=2\\
(\Z/2\Z) \times (\Z/2^{\alpha-2}\Z) & \alpha \ge 3.
\end{cases}
\]
The lemma follows.
\end{proof}
\begin{lem}\label{lem:CRT}
If $m_1$, $m_2$ are positive integers and $\gcd(m_1,m_2)=1$
then
\[
	r_\ell(m_1 m_2)=r_\ell(m_1)+r_\ell(m_2).
\]
\end{lem}
\begin{proof}
By the Chinese Remainder Theorem, $(\Z/m_1 m_2 \Z)^\times \cong (\Z/m_1 \Z)^\times \times
	(\Z/m_2 \Z)^\times$. The lemma follows.
\end{proof}
\begin{lem}\label{lem:Nn}
Let $n$ be the conductor of a cyclic
field of degree $\ell$. Then
\begin{equation}\label{eqn:condfact}
n=\ell^v \cdot \prod_{i=1}^t q_i
\end{equation}
where $q_1,\dotsc,q_t$ are distinct primes
$\equiv 1 \pmod{\ell}$ and 
	\begin{equation}\label{eqn:vee}
v=\begin{cases}
\text{$0$ or $2$} & \text{if $\ell \ne 2$}\\
\text{$0$, $2$ or $3$} & \text{if $\ell=2$}.
\end{cases}
	\end{equation}
Moreover,
\[
N(n)=
\begin{cases}
(\ell-1)^{t-1} & \text{if $v=0$}\\
(\ell-1)^t & \text{if $v=2$}\\
\ell (\ell-1)^t & \text{if $\ell=2$ and $v=3$}.
\end{cases}
\]
\end{lem}
\begin{proof}
Applying M\"{o}bius inversion to \eqref{eqn:condsum}
we have
\[
N(n)=\sum_{d \mid n} 
\mu\left(\frac{n}{d}\right) \cdot M(d).
\]
From Lemma~\ref{lem:Mn}, and using the fact that $\sum_{d \mid n}
\mu(n/d)=0$ for $n>1$ we have
\begin{equation}\label{eqn:conv}
N(n)=\frac{1}{\ell-1} \sum_{d \mid n}
\mu\left(\frac{n}{d}\right) \cdot \ell^{r_\ell(d)}.
\end{equation}
Now the function $g(m):= \ell^{r_\ell(m)}$ 
is multiplicative by Lemma~\ref{lem:CRT}.
Therefore the convolution $\mu * g$ is 
also multiplicative. 
Note that  \eqref{eqn:conv} 
may be re-expressed as $(\ell-1) N(n)=(\mu * g)(n)$.
Thus
\[
(\ell -1 ) N(n) \; = \; \prod_{q^\alpha \mid \mid n} (\mu *g)(q^\alpha),
\]
where the product is taken over prime powers $q^\alpha$
dividing $n$ exactly. In particular, since $n$
is the conductor of a cyclic degree $\ell$ field, $N(n) \ne 0$,
and so $(\mu *g)(q^\alpha) \ne 0$ for all $q^\alpha \mid \mid n$.

Now let $q \ne \ell$ and $\alpha \ge 1$. 
Then
\[
(\mu*g)(q^\alpha)=\ell^{r_\ell(q^{\alpha})}-\ell^{r_\ell(q^{\alpha-1})}=
\begin{cases}
\ell -1 & \text{if $q \equiv 1 \pmod{\ell}$ and $\alpha=1$,}\\
0 & \text{if $q \not\equiv 1 \pmod{\ell}$ or $\alpha \ge 2$}\\
\end{cases}
\]
by Lemma~\ref{lem:qalpha}. 
It follows that $n$
satisfies \eqref{eqn:condfact} where the $q_i$
are distinct primes $\equiv 1 \pmod{\ell}$ and
that
\[
N(n)=(\ell-1)^{t-1} \cdot (\mu*g)(\ell^v).
\]
Finally
\[
(\mu*g)(\ell^v) \; = \; \begin{cases}
1 & \text{if $v=0$}\\
\ell-1 & \text{if $v=2$}\\
\ell^2-\ell & \text{if $\ell=2$ and $v=3$}\\
0 & \text{in all other cases,}\\
\end{cases}
\]
again from Lemma~\ref{lem:qalpha}. This completes the proof.
\end{proof}

\begin{lem}\label{lem:Ikehara}
Let $\ell$ be a prime.
Let $\cP \subseteq \PP$ be regular of density $\alpha>0$.
Suppose that all primes in $\cP$
are $\equiv 1 \pmod{\ell}$.
Let $\cB$ be the set of all squarefree positive
integers with prime divisors belonging entirely to $\cP$.
Denote by $\omega(n)$ the number of distinct prime
divisors of an integer $n$.
Then there is some $c>0$ such that
\[
\sum_{\substack{n \in \cB \\ n \le X}} (\ell-1)^{\omega(n)}
\;
\thicksim
\;
c \cdot \frac{X}{(\log{X})^{1-\beta}}
\]
as $X \rightarrow \infty$, where $\beta=\alpha \cdot (\ell-1)$.
\end{lem}
\begin{proof}
Consider the Dirichlet series
\[
D(s) :=  \sum_{n \in \cB} \frac{(\ell-1)^{\omega(n)}}{n^s}=\prod_{p \in \cP} \left(1+\frac{\ell-1}{p^s} \right).
\]
Then
\[
\log{D(s)}=\sum_{p \in \cP} \frac{\ell-1}{p^s} + \theta(s)
\]
where $\theta$ is holomorphic on $\re(s) > 1/2$. 
By \eqref{eqn:regular},
\[
\log{D(s)}=\beta \cdot \log\left( \frac{1}{s-1} \right) \; + \; \phi(s)
\]
and $\phi$ is holomorphic on $\re(s) \ge 1$. Thus
\[
D(s)=\frac{\Phi(s)}{(s-1)^\beta}
\]
where $\Phi(s)=\exp(\phi(s))$ is holomorphic and non-zero
on $\re(s) \ge 1$. Since $\cP$ is contained
in the set of primes $\equiv 1 \pmod{\ell}$
we know that $0<\alpha \le 1/(\ell-1)$,
and so $0< \beta \le 1$.

We now apply to $D(s)$ a variant of  Ikehara's Tauberian theorem
due to Delange \cite[Theorem 7.28]{Tenenbaum} to obtain
\[
\sum_{\substack{n \in \cB\\ n \le X}} (\ell-1)^{\omega(n)} \;
\thicksim \;
\frac{\Phi(1)}{\Gamma(\beta)} \cdot \frac{X}{(\log{X})^{1-\beta}},
\]
where $\Gamma$ denotes the gamma function.
The lemma follows.
\end{proof}
\begin{proof}[Proof of Proposition~\ref{prop:Ikehara}]
Suppose first that $\ell \notin \cP$,
and let $\cB$ be as in the statement
of Lemma~\ref{lem:Ikehara}.
Then, by Lemma~\ref{lem:Nn},
\[
\#\cF^{\mathrm{cyc}}_{\cP,\ell}(X) \;=\; \sum_{\substack{n \in \cB\\ n \le X}} N(n)
\; = \; \frac{1}{\ell-1}\sum_{\substack{n \in \cB\\ n \le X}} (\ell-1)^{\omega(n)}.
\]
The proposition follows immediately from Lemma~\ref{lem:Ikehara}
in this case.
%
Suppose next that $\ell \in \cP$ and $\ell \ne 2$.
Let $\cP^\prime=\cP\setminus \{\ell\}$
and now let $\cB$ be the set of all squarefree positive
integers with prime divisors belonging entirely to $\cP^\prime$.
By Lemma~\ref{lem:Nn}
\[
	\# \cF_{\cP,\ell}^{\mathrm{cyc}}(X)= 
\sum_{\substack{n \in \cB \\ n \le X}} N(n)
\; + \;
\sum_{\substack{n \in \cB \\ n \le X/\ell^2}} N(\ell^2 n)
=
\sum_{\substack{n \in \cB \\ n \le X}} (\ell-1)^{\omega(n)-1} 
\; + \;
\sum_{\substack{n \in \cB \\ n \le X/\ell^2}} (\ell-1)^{\omega(n)}. 
\]
The proposition follows from Lemma~\ref{lem:Ikehara} 
in this case also.
The case $\ell=2 \in \cP$ is dealt with similarly.
\end{proof}

\section{Proof of Theorem~\ref{thm:cyclic}}\label{sec:proofthmcyclic}
Let $\ell$ be a rational prime, and let
$A/\Q$ be an abelian variety.
The following result is stated as an exercise in \cite[Section 4.6]{SerreMW}.
\begin{lem}\label{lem:WeilPairing}
$\Q(\zeta_\ell) \subseteq \Q(A[\ell])$.
\end{lem}
\begin{proof}
If $A$ is principally polarized then the lemma
is a famous consequence of the properties of the Weil pairing on $A[\ell]$.
We learned the following more general argument
from a \texttt{Mathoverflow} post by Yuri Zarhin
\cite{Zarhin}.
Write $A^\vee$
for the dual abelian variety,
and let $\phi \; : \; A \rightarrow A^\vee$
be a $\Q$-polarization of smallest possible degree.
If $A[\ell] \subseteq \ker(\phi)$, then $P \mapsto \phi((1/\ell)P)$
is a well-defined $\Q$-polarization
contradicting the minimality of the degree. Thus
there is some $Q \in A[\ell]$
such that $\phi(Q) \in A^\vee[\ell] \setminus \{0\}$.
The non-degenerancy of the Weil pairing
$e_\ell : A[\ell] \times A^\vee[\ell] \rightarrow \langle \zeta_\ell \rangle$ ensures the existence of $P \in A[\ell]$ such that
$e_\ell(P,\phi(Q))=\zeta_\ell$. Now  
	$P$ and $\phi(Q)$ are fixed by $\Gal(\overline{\Q}/\Q(A[\ell]))$,
and so, by the Galois-compatibility of the Weil pairing,
$\zeta_\ell$ is also fixed by $\Gal(\overline{\Q}/\Q(A[\ell]))$. 
Thus $\zeta_\ell \in \Q(A[\ell])$. 
\end{proof}
We let $G_\ell(A)$, $H_\ell(A)$ be as in \eqref{eqn:GH},
and $\cC_\ell(A)$ as in \eqref{eqn:cC}.
We note that $\cC_\ell(A)$ is a finite union
of conjugacy classes.
We now suppose that $A$ and $\ell$ satisfy the hypotheses
of Theorem~\ref{thm:cyclic},
namely
\begin{enumerate}[(i)]
\item $A(\Q)=0$;
\item $\cC_\ell(A) \ne \emptyset$.
\end{enumerate}
Let $S$ be a finite set of rational primes.
Enlarge $S$ so that includes $\ell$ and all the primes
of bad reduction for $A$. Let $\PP_\ell$ 
be as in \eqref{eqn:PPell}.
Let
\[
	\cP=\{p \in \PP_\ell \; : \; 
	\text{$p \in S$ or $\sigma_p \notin \cC_\ell(A)$}\};
\]
here, as in Lemma~\ref{lem:equiv},
$\sigma_p \in G_\ell(A)$ denotes a Frobenius
element associated to $p$.
\begin{lem}\label{lem:Frobenian}
The set $\cP$ is Frobenian (and therefore regular)
of density
	\begin{equation}\label{eqn:density}
		\alpha:=\frac{\# H_\ell(A) - \# \cC_\ell(A)}{(\ell-1) \cdot \# H_\ell(A)}.
	\end{equation}
\end{lem}
\begin{proof}
Let $p$ be a sufficiently large prime.
	By part (a) of Lemma~\ref{lem:equiv},
	we have $p \in \cP$ if and only if
	$\sigma_p \in H_\ell(A) \setminus \cC_\ell(A)$.
	Thus $\cP$ is Frobenian of density
	\[
		\frac{\# H_\ell(A) - \# \cC_\ell(A)}{\# G_\ell(A)}.
	\]
	The lemma follows as $G_\ell(A)/H_\ell(A) \cong \Gal(\Q(\zeta_\ell)/\Q)$
	has order $\ell-1$.
\end{proof}

\begin{lem}\label{lem:integral}
Let $L/\Q$ be cyclic of degree $\ell$
and suppose $(A\setminus 0)(\OO_{L,S}) \ne \emptyset$.
Then the conductor of $L$ is divisible only by primes
belonging to $\cP$.
\end{lem}
\begin{proof}
We know from Lemma~\ref{lem:Nn} that the 
prime divisors of the conductor of $L$
belong to $\PP_\ell$.
Let $p \equiv 1 \pmod{\ell}$
be a prime of good reduction for $A$
dividing the conductor of $L$.
It is sufficient to show that $\sigma_p \notin \cC_\ell(A)$.
Suppose $\sigma_p \in \cC_\ell(A)$. 
Since $p$ divides the conductor of $L$ it is ramified
in $L$. However, $\Gal(L/\Q)$ is cyclic of order $\ell$.
As the inertia subgroup at $p$ is non-trivial
it must equal $\Gal(L/\Q)$. We deduce
that $p$ is totally ramified in $L$.
	Also, by Lemma~\ref{lem:equiv},
we have $\ell \nmid \#A(\F_p)$. Recall that $A(\Q)=0$
by assumption (i) above. We now apply Theorem~\ref{thm:main}
to conclude that $(A \setminus 0)(\OO_{L,S})=\emptyset$,
giving a contradiction.
\end{proof}

\subsection*{Proof of Theorem~\ref{thm:cyclic}}
By assumption (ii) above $\cC_\ell(A) \ne \emptyset$.
It follows from \eqref{eqn:density}
that
	$\alpha  < 1/(\ell-1)$.
	Moreover, from the definition of $\cC_\ell(A)$
	in \eqref{eqn:cC}, we note that
	$1 \in H_\ell(A)$ but $1 \notin \cC_\ell(A)$.
	It follows that $\alpha>0$.
	Lemma~\ref{lem:Frobenian} tells us that $\cP$
	is regular of density $\alpha$.
By Lemma~\ref{lem:integral},
\[
\{L \in \cF^{\mathrm{cyc}}_\ell(X) \; : \; (A\setminus 0)(\OO_L) \ne \emptyset\}
\; \subseteq \; \cF^{\mathrm{cyc}}_{\cP,\ell}(X),
\]
where $\cF^{\mathrm{cyc}}_{\cP,\ell}(X)$
is defined in Proposition~\ref{prop:Ikehara}.
By Proposition~\ref{prop:Ikehara} (see also the remark following
that proposition), there are $c_1$, $c_2>0$ such that
\[
\# \cF^{\mathrm{cyc}}_{\cP,\ell}(X) \, \thicksim \,
c_1 \cdot \frac{X}{(\log{X})^{1-\beta}}, \qquad
\# \cF^{\mathrm{cyc}}_{\ell}(X) \, \thicksim \,
c_2 \cdot X
\]
as $X \rightarrow \infty$, where 
\[
	\beta \; =\; (\ell-1)\alpha \; = \; 
		\frac{\# H_\ell(A) - \# \cC_\ell(A)}{\# H_\ell(A)}.
\]
This proves the theorem.

\bibliographystyle{abbrv}
\bibliography{samir}

\end{document}